\documentclass[12pt,a4article]{article}
\usepackage[latin1]{inputenc}
\usepackage{amsthm, amssymb, amscd,amsmath, amsfonts}
\usepackage{color}
\usepackage{graphicx}
\usepackage{hyperref}
\usepackage{authblk} 

\hypersetup{colorlinks=true, citecolor=blue, urlcolor=blue,linkcolor=blue}

\newtheorem{theorem}{Theorem}
\newtheorem{lemma}[theorem]{Lemma}
\newtheorem{proposition}[theorem]{Proposition}
\newtheorem{corollary}[theorem]{Corollary}
\newtheorem{remark}[theorem]{Remark}

\newcommand{\sinc}{\text{sc}}

\newcommand{\Prob}{\mathbb P}
\newcommand{\E}{\mathbb E}
\newcommand{\eps}{\varepsilon}
\def\N{\mathbb N}
\def\R{\mathbb R}

\def\indicator{\mathbf{1}}

\author{Jean-Marc Aza\"{i}s
\qquad 
Federico Dalmao
\qquad
Jos\'e R. Le\'on\and
Ivan Nourdin
\qquad
Guillaume Poly
}

\title{Local universality of the number of zeros of random trigonometric polynomials with continuous coefficients}
\begin{document}
\maketitle
\begin{abstract}
Let $X_N$ be a random trigonometric polynomial of degree $N$
with iid coefficients and
let $Z_N(I)$ denote the (random) number of its zeros  lying in the compact interval $I\subset\R$. 
Recently, a number of important advances were made in the understanding 
of the asymptotic behaviour of $Z_N(I)$ as $N\to\infty$,  in the case of standard Gaussian coefficients.
The main theorem of the present paper is a universality result, that states that
the limit of $Z_N(I)$ does not really depend on the exact distribution of the coefficients of $X_N$. More precisely, assuming that these latter are iid with mean zero and unit variance and have a density satisfying certain conditions, we show that $Z_N(I)$
converges in distribution toward $Z(I)$, the number of zeros 
within $I$ of the centered stationary Gaussian process admitting the cardinal sine for covariance function.
\end{abstract}

\section{Introduction and main result}

Random polynomials are popular models in probability theory. They have found a lot of applications in several fields of physics, engineering and economics. 
In particular, there is a great variety of problems where the distribution of {\it zeros} of random polynomials occurs, including nuclear physics (in particular, random matrix theory), statistical mechanics or quantum mechanics, to name but a few; see, e.g., Bharucha-Reid and Sambandham \cite{BR} or Bogomolny, Bohigas and Leb{\oe}uf \cite{BBL} and references therein.

The most studied classes of random polynomials are the algebraic and trigonometric ensembles.  
As a matter of fact, 
it was rapidly observed that the behaviour of their zeros
exhibit important differences. 
For instance, both
the asymptotic mean and variance of the number of real roots of Kac algebraic polynomials 
are equivalent to $\log N$, while in the trigonometric case both the
asymptotic mean and variance of the number of  roots on $[0,2\pi]$ are equivalent to $N$, with $N$ the degree of the polynomial. See, e.g., \cite{granville} or \cite{far98} for precise statements and references.
Besides, 
for algebraic polynomials it often happens that the dominant term in its expansion is solely responsible for the limit; in contrast,  
in the trigonometric case generally each term  contributes infinitesimally.

One can find in \cite{BBL}
several reasons explaining 
the wide interest of scientists in random trigonometric polynomials.
For instance, in the quantum semiclassical limit 
one expects to have a large proportion of roots 
on, or close to, the unit circle in the complex plane. 
Under a certain natural sufficient condition 
on the coefficients of the random polynomials 
(self-invertibility), the authors were led to consider trigonometric polynomials.  

More specifically, throughout this paper we will deal with random trigonometric polynomials of the form
\begin{eqnarray}
P_N(t) = \sum_{n=1}^N 
\big\{
a_n\,
\cos(nt) + b_n\,
\sin(nt)
\big\} ,\quad t\in\R,
\label{P_N}
\end{eqnarray}
where the coefficients $a_n$ and $b_n$ are iid random variables that are normalised so that $\E[a_1]=0$ and $\E[a_1^2]=1$. 
The problem we want to study is the following:\\

{\bf Question Q}. {\it Fix  a small interval containing $t=0$. How does the number of zeros of $P_N$ lying in this interval behave as $N\to\infty$}?\\

In order to solve this question, we first have to find the right scale at which a non-degenerate limit may happen. 
This leads us to change $t$ into $\frac{t}{N}$ and to consider the following normalized version of $P_N$:
\begin{eqnarray}
X_N(t) = \frac{1}{\sqrt{N}}\sum_{n=1}^N 
\bigg\{
a_n\,
\cos\left(\frac{nt}{N}\right) + b_n\,
\sin\left(\frac{nt}{N}\right)
\bigg\} ,\quad t\in\R.
\label{eq:trig-pol}
\end{eqnarray}
We can now investigate the limit of the number of the zeros of $X_N$  lying in any compact interval $I\subset\R$. 
(Observe that the factor $1/\sqrt{N}$ in (\ref{eq:trig-pol}) is of course totally useless as far as zeros are concerned; but since it will play a role when passing to the limit later on, we found convenient to keep it in our definition of $X_N$.)

In the existing literature, a lot of investigations about the number of zeros of (\ref{P_N}) concerns the particular case where $a_1\sim N(0,1)$. This situation turns out to be more amenable to analysis;
indeed, $P_N$ (or $X_N$) is then Gaussian, centered and stationary. For instance, one can rely on 
 the Rice's formula to find that
the mean of the number $Z_N(I)$ of zeros of $X_N$ within the compact interval $I$ converges to $|I|/(\pi\sqrt{3})$, with $|I|$ the length of $I$. (We recall from the celebrated Gaussian Rice's formula that, for any centered stationary Gaussian process $X$ with variance 1, the mean of the number of zeros within any interval $I$ is given by $\frac{\sqrt{-r''(0)}}{\pi}|I|$, where $r(t-s)=\E[X(t)X(s)]$). 
Nevertheless, even in this Gaussian framework the analysis becomes increasingly harder when higher moments are concerned. For example, a prediction for the limit of the variance of the number of zeros was made in \cite{BBL} in 1996, but it was only a dozen years later that this claim was confirmed by Granville and Wigman \cite{granville} by combining techniques from probability theory, stochastic processes and harmonic analysis. 
Besides, the authors of \cite{granville} also showed that a central limit theorem (CLT) for the number of zeros of $X_N$ within $[0,2\pi N]$ holds.
Finally, we would like to conclude this very short picture of the existing results in the case $a_1\sim N(0,1)$ by mentioning the recent paper \cite{al} by Aza\"is and Le\'on, in which the authors make use of the Wiener chaos techniques to prove, more generally, a CLT for the number of crossings of any given level $u\in\R$ (see also \cite{adl} for a related study).

As we just explained, assuming  in (\ref{P_N}) or (\ref{eq:trig-pol}) that the coefficients of $X_N$ are Gaussian is of great help when dealing with the moments of the number of zeros, as it gives one access to a variety of tools and desirable properties.
In contrast, solving Question Q when $a_1\sim\frac{1}{2}(\delta_1+\delta_{-1})$ (that is, in the case where $a_1$ is distributed according to the Rademacher distribution) seems clearly out of reach of existing methods. However, 
empirical simulations (see Figure \ref{fig}) suggest that the number $Z_N$ of zeros of $X_N$
within any given compact interval $I$ exhibits a {\it universality phenomenon}; this leads us to formulate the natural following conjecture.\\

{\bf Conjecture C}.
{\it Assume that $a_1$ is square integrable with mean zero and unit variance.
Then the number of zeros of $X_N$ within any given compact interval $I\subset\R$ converges, as $N\to\infty$, to  the number of zeros 
of the centered stationary Gaussian process admitting the cardinal sine for covariance function, and this \underline{irrespective of the exact distribution of $a_1$}.\\
}

\begin{figure}[!h]\label{fig}
\begin{center}
\includegraphics[scale=.75,trim=2cm 7cm 0 7cm]{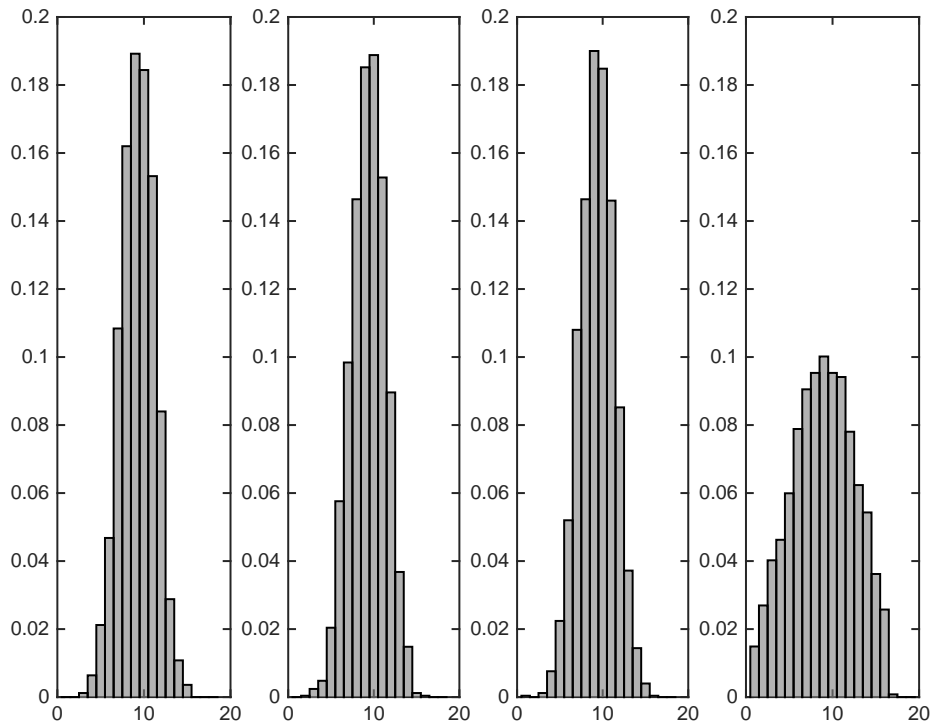}
\end{center}
\caption{\it Empirical distribution of the number of zeros of $X_{50}$ on the interval $[0,50]$, in four situations for the distribution of $a_1$. 
{\rm First}: Rademacher case. 
{\rm Second}: Uniform case on [-1,1] (scaled to have variance 1).
{\rm Third}: Gaussian case. 
{\rm Fourth}: Cauchy case.}
\end{figure}

At this stage, it is worth stressing that our universality conjecture concerns the {\it local} behavior of the number of zeros. Indeed, we are interested in the number of zeros of $X_N$ (not $P_N$) on a compact interval $I$ (for instance, $I=[0,50]$ like in Figure \ref{fig}). Another natural problem is to rather consider the {\it global} behavior, by looking this time at the number of zeros of $P_N$ on a compact interval $I$ or, equivalently, at the number of zeros of $X_N$ on $nI$. We refer to
the recent preprint \cite{ap} by Angst and Poly, for an analysis of the universality phenomenon for the {\it mean} number of zeros of $X_N$ on the interval $[0,n\pi]$.

Let us go back to the present paper.
Our main result, Theorem \ref{th} just below, provides a positive solution to our Conjecture C, in the case of {\it continuous} coefficients whose density satisfies some regularity and boundedness conditions.

\begin{theorem}\label{th}
Assume in (\ref{eq:trig-pol}) that $a_n,b_n$ are iid square integrable random variables and admit a density $\rho:\R\to (0,\infty)$
of the form $\rho=e^{-\Psi}$, and  that $\E[a_1]=0$ and $\E[a_1^2]=1$.
Assume furthermore that $\Psi$ is $C^{\infty}$ and satisfies
that
\begin{equation}\label{Hypo-densite}
\Psi^{(p)} \in \bigcap_{q\geq 1} L^q(e^{-\Psi(x)} dx),
\quad p\geq 1.
\end{equation}
Then, for any compact interval $I\subset\R$,
\begin{equation}\label{cv}
Z_{X_N}(I) \overset{\rm law}{\to} Z_W(I)\quad\mbox{as $N\to\infty$},
\end{equation}
 where, whenever $Y$ is a process, we denote by $Z_Y(I)$ the number of zeros of $Y$ within the interval $I$
and where $W$ is the centered stationary Gaussian process 
defined on $[0,\infty)$ with covariance function 
\begin{equation}\label{r}
\E[W(s)W(t)]=\sinc(t-s):=\frac{\sin(t-s)}{t-s}.
\end{equation}  
\end{theorem}
\begin{remark}
{\rm
As this stage, we would like to emphasize that the assumption (\ref{Hypo-densite}) is  general and actually contains a wide range of densities. Indeed, assume for instance the following two conditions on $\Psi$, that 
roughly express the fact that $\Psi$ diverges at a polynomial speed and that its derivatives have polynomial growths: 
\begin{itemize}
\item their exist $\alpha,c,M>0$ such that  
$\Psi(x) \ge c |x|^{\alpha}$ for all $|x|>M$;
\item for each $p\ge 0$, their exist $\beta_p,c_p >0$ such that $\left|\Psi^{(p)}(x)\right|\le c_p\left(1+|x|^{\beta_p}\right)$.
\end{itemize}
Then condition (\ref{Hypo-densite}) holds. Also, it is worth highlighting that densities of the form $e^{-\Psi}$ appear naturally as invariant measures of diffusions of the form $dX_t=d W_t-\Psi'(t) dt$. Indeed, the infinitesimal generator is $\frac12 \frac{d^2}{dx^2}-\Psi'(x) \frac{d}{dx}$.
}
\end{remark}

The heuristic behind Theorem \ref{th} is pretty simple. Using the classical CLT, it is easy to see that the finite dimensional marginals of  $X_N$ converge to those of $W$. Thus, expecting that (\ref{cv}) holds true looks reasonable.
However, due to the highly non-linear structure of the problem as well as the complexity of the relationships between zeros and coefficients, to transform this intuition  into a rigorous theorem is a challenging question.
To reach the conclusion of Theorem \ref{th}, we will use techniques from two apparently distinct fields, namely
Rice's formulas (which have been used extensively to study the zeros of random polynomials) 
 on one hand and local central limit theorems deduced by Malliavin calculus on the other one. 
The required smoothness and decay of the finite-dimensional densities of the process involved in the Rice's formulas will
be obtained by introducing a suitable formalism of integration by parts and will require a long and careful analysis.
Note that the idea of combining Rice and Malliavin techniques appeared for the first time in the work by Nualart and Wschebor \cite{nw}.

Let us be a little bit more precise on the technical details, by sketching the route we will follow in order to prove Theorem \ref{th}.
We will first check in Lemma \ref{prop:det-moments} 
that the distribution of $Z_W(I)$ is characterised by its moments.
As a result, in order to reach the conclusion of Theorem \ref{th} it will be enough to check that 
all the moments of $Z_{X_N}(I)$ converge, as $N\to\infty$, 
to the corresponding moments of $Z_W(I)$. 
For technical reasons, it will be convenient (as well as equivalent) to rather prove 
the convergence of {\it factorial} moments. 
In other words, we will  show that, 
for all integer $m\geq 1$, 
\begin{equation}\label{eq:conv-moments}
\E[Z_{X_N}(I) ^{[m]}]\to \E[Z_{W}(I) ^{[m]}]\, \mbox{ as } N\to\infty,
\end{equation}
where  $x^{[m]}=x(x-1)\ldots(x-m+1)$. 
The proof of (\ref{eq:conv-moments}) shall be done into two main steps of totally different natures:
\begin{enumerate}
\item[(i)]
Firstly, by means of the Rice's formula
we will give integral expressions for the factorial moments. To describe these expressions, we need to introduce some notation. Fix an integer $m\geq 1$ and consider, for $\eps>0$, ${\mathbf  t}=(t_1,\dots,t_m)\in I^m$
and ${\mathbf  x}=(x_1,\dots,x_m),{\mathbf  x}^\prime=(x^\prime_1,\dots,x^\prime_m)\in \R^m$,
\begin{itemize}
\item the set $D_m$ of hyper-diagonals,
\begin{equation}\label{hyper}
D_m=\{(t_1,\dots,t_m)\in I^m:\exists i\neq j / t_i=t_j\};
\end{equation}
\item the $\eps$-enlargement of $D_m$,
\begin{equation}\label{enlarge}
D^{\,\eps}_m=\{(t_1,\dots,t_m)\in I^m:\exists i\neq j / |t_i-t_j|<\eps\};
\end{equation} 
 \item the random vector
 \begin{equation}\label{vn}
V_N({\mathbf  t})=(X_N(t_1),\dots,X_N(t_m),X^{\prime}_N(t_1),\dots,X^{\prime}_N(t_m));
\end{equation}
\item and  (whenever it exists)  the density $p_{N,{\mathbf  t}}({\mathbf  x};{\mathbf  x}^\prime)$ of $V_N ({\bf t})$ at $({\mathbf  x},{\mathbf  x}')$.
\end{itemize}
Corollary \ref{cor:Rice-trunc} will basically state that, for $N$ large enough,
\begin{equation}\label{eq:Rice2}
\E[Z_{X_N}(I) ^{[m]}]
=\int_{I^m\setminus D^{\eps}_m}
\int_{\R^m}|y_1|\ldots |y_m | \,p_{N,{\mathbf  t}}
({\mathbf  0};{\mathbf  y})
d{\mathbf  y} d{\mathbf  t}+ O(\eps^{1/5}),
\end{equation}
where  ${\boldsymbol 0}=(0,\dots,0)\in\R^m$ and 
where the constant involved in the Landau notation  $O(\cdot)$
 is uniform with respect to $N$.
\item[(ii)]
Secondly, Proposition \ref{prop:conv} will establish a local limit theorem 
for the density $p_{N,{\mathbf  t}}$ appearing in \eqref{eq:Rice2}.
By passing to the limits as $N\to\infty$ (after swaping the limit and the integral by dominated convergence) and $\varepsilon\to0$,
 the conclusion of Theorem \ref{th} will follow.
\end{enumerate}

Throughout all the paper $(Const)$ denotes an unimportant universal constant whose value may change from one occurrence to another. When a constant depends  on some parameter, for example $\eps$,  we shall  denote it by $ C_\eps$.

The rest of the paper is organized as follows. 
The proof of Theorem \ref{th} is given in Section 2, whereas Section 3 gathers the most technical results. 

\section{Proof of Theorem \ref{th}}

Theorem \ref{th} is shown by means of the method of moments. 
More precisely, its proof is splitted into two steps:
\begin{enumerate}
\item Section \ref{S1} will first show that 
the distribution of $Z_W(I)$ is determined by its moments.
\item Section \ref{S2} will then show that the (factorial) moments
of $Z_{X_N}(I)$ converges to  the corresponding moments of $Z_W(I)$. 
\end{enumerate}

\subsection{The distribution of $Z_W(I)$ is determined by its moments}
\label{S1}
Recall from Theorem \ref{th} that $W$ is the centered stationary Gaussian process on $[0,\infty)$ admitting the cardinal sine for covariance function. 

\begin{lemma}\label{prop:det-moments}
All the moments of $Z_W(I)$ are finite. 
Furthermore, the law of $Z_W(I)$ is determined by its moments.
\end{lemma}
\begin{proof}
Using Nualart-Wschebor criterion (see \cite[Theorem 3.6, Corollary 3.7]{aw}), 
one has that all the moments of $Z_W(I)$ are indeed finite.
In order to show that the moments of $Z_W(I)$ determine its law, we use the well-known 
Carleman's condition, that we restate here for convenience.
\begin{lemma}[Carleman]
Let $Y$ be a real-valued random variable. 
If $\mu_m=\E(|Y|^m)<\infty$ for all $m\in\N$ 
and if
$$
\sum^\infty_{m=1} \mu_{2m}^{-\frac{1}{2m}}=\infty,
$$
then the law of $Y$ is determined by its moments.
\end{lemma}

Assume that the length $|I|$ of $I$ is greater than 1, otherwise let us bound it by $1$. 
Set $\mu_{m,N}=\E(Z_{X_N}(I)^m)$ and $\mu_{m}=\E(Z_W(I)^m)$. 
Using inequality \eqref{ineq1} below, we have
$$
\mu_{m,N}\leq (Const)\sum^\infty_{k=1}k^m\frac{|I|^{k-1/2}}{\sqrt{k!(k-1)!}}.
$$
Proposition \ref{prop:conv} below implies that
$\mu_m$ is the limit of $\mu_{m,N}$ as $N\to\infty$. As a consequence, 
$$
\mu_{m}\leq (Const)\sum^\infty_{k=1}k^m\frac{|I|^{k-1/2}}{\sqrt{k!(k-1)!}}.
$$
Now, if $k\leq 3m$ we use the rough bound
$$
k^m\frac{|I|^{k-1/2}}{\sqrt{k!(k-1)!}}\leq |I|^{-1/2}(3m|I|^{3})^m.
$$
If $k\geq 3m$, we use the bound
\begin{equation*}
k^m\frac{|I|^{k-1/2}}{\sqrt{k!(k-1)!}}
\leq \frac{k^m|I|^{k-1/2}}{(k-2m)^{2m-1/2}(k-2m)!}\\
\leq \frac{3^m|I|^{k-1/2}}{(k-2m)^{m-1/2}(k-2m)!}.
\end{equation*}
Hence,
$$
\mu_{m}\leq (3m)^{m+1}\,|I|^{3m-1/2}
+3^{m}|I|^{3m-1/2}R(m),
$$
where $R(m)=\sum^\infty_{k=0}\tfrac{|I|^{k}}{(k+m)^{m-1/2}(k+m)!}$ is bounded and decreasing with $m$. 
Thus,
$$
\mu_m\leq (Const)(3m|I|^3)^{m+1}.
$$
Therefore,
$$
\sum^\infty_{m=1}\mu^{-\frac{1}{2m}}_{2m}\geq
\sum^\infty_{m=1}\frac{1}{(3|I|^{3}2m)^{(2m+1)/2m}},
$$
which is indeed divergent. 
The desired result follows.
\end{proof}
 
 \subsection{Convergence of factorial moments}\label{S2}

We start with a basic fact about the non-degeneracy of the finite-dimensional distributions of the process $W$
appearing in Theorem \ref{th}. Recall also $D_m$ from (\ref{hyper}).
\begin{lemma}\label{lemma:non-deg}
For $\mathbf t=(t_1,\ldots,t_m) \in I^m$, let us consider the covariance matrix $\Sigma({\bf t})$ of the Gaussian vector 
\begin{equation*}
V(\mathbf t)=(W(t_1),\ldots,W(t_m);W'(t_1),\ldots,W'(t_m)).
\end{equation*}
If $t\not\in  D_m$
then $\Sigma({\bf t})$ is invertible. 
\end{lemma}

\begin{proof}
We shall use the method of Cram\'{e}r and Leadbetter   \cite{cl}, 
see also Exercise 3.4 in Aza\"\i s and Wschebor  \cite{aw}. 
Note that the spectral density of $X$ is $ f(\lambda)=\frac12\,\mathbf 1_{[-1,1]}(\lambda)$. 
We want to study the strict positiveness of
$F({\mathbf z})={\mathbf z}^T\,\Sigma\,{\mathbf z},$ 
where
$$
\mathbf z
=(\mathbf z_1,\mathbf z_2)=(z^1_1,z^1_2,\ldots,z^1_m;z^2_1,z^2_2,\ldots,z^2_m).
$$
With obvious notation, it holds that
$$
F({\mathbf z})= \mathbf z_1^T\,\Sigma_{11}\,\mathbf z_1+\mathbf z_1^T\,\Sigma_{12}\, {\bf z}_2+\ {\bf z}_2^T\,\Sigma_{21}\,\mathbf z_1+\mathbf z_2^T\,\Sigma_{22}\,\mathbf z_2.
$$
We have
\begin{eqnarray*}
2\,\mathbf z_1^T\,\Sigma_{11}\,\mathbf z_1&=&\int_{-1}^1\sum_{j=1}^m\sum_{l=1}^m e^{i(t_l-t_j)\lambda}z_l^1z_j^1d\lambda=\int_{-1}^1\left|\sum_{l=1}^de^{i t_l\lambda}z^1_l\right|^2 d\lambda\\
&=:&
\int_{-1}^1|P_1( \mathbf t,\mathbf z_1,\lambda)|^2 d\lambda,\\
2\,\mathbf z_2^T\,\Sigma_{22}\,\mathbf z_2&=&\int_{-1}^1\sum_{j=1}^m\sum_{l=1}^m e^{i(t_l-t_j)\lambda}z_l^2z_j^2\lambda^2 d\lambda=:\int_{-1}^1|P_2( \mathbf t,\mathbf z_2,\lambda)|^2\lambda^2 d\lambda,\\
2\,\mathbf z_1^T\,\Sigma_{12}\,\mathbf z_2&=&\int_{-1}^1\sum_{j=1}^m\sum_{l=1}^m e^{i(t_l-t_j)\lambda}z_l^1z_j^2 i\lambda d \lambda
=\int_{-1}^1 P_1( \mathbf t,\mathbf z_1,\lambda)
\overline{P_2( \mathbf t,\mathbf z_2,\lambda)}i\lambda d\lambda,\\
2\,\mathbf z_2^T\,\Sigma_{21}\,\mathbf z_1&=&-\int_{-1}^1 P_2( \mathbf t,\mathbf z_2,\lambda)
\overline{P_1( \mathbf t,\mathbf z_1,\lambda)}i\lambda d\lambda.
\end{eqnarray*}
As a result,
$$
F(\mathbf z)=\frac12\int_{-1}^1 |P_1( \mathbf t,\mathbf z_1,\lambda)-iP_2( \mathbf t,\mathbf z_2,\lambda)\lambda|^2 d\lambda.
$$ 
Since the analytic function inside the square in the right-hand side of the previous identity cannot vanish almost everywhere
when the $t_i$'s are pairwise distinct and ${\bf z}\neq 0$, we deduce that
$F({\bf z})>0$ for all ${\bf z}\neq 0$. This is the desired conclusion. 
\end{proof}

Apart from the most technical proofs that are postponed in Section 3, we are now in a position to check the convergence of moments.
\begin{proposition}\label{prop:conv}
For all $m\in\N$, one has
\begin{equation*}
\E[Z_{X_N}(I)^{[m]}]\to \E[Z_{W}(I)^{[m]}]\quad \mbox{as $N\to\infty.$}
\end{equation*}
\end{proposition}

\begin{proof}
Fix $m\in\mathbb{N}^*$ and a compact interval $I\subset\R$. By the forthcoming Corollary \ref{cor:Rice-trunc} we know that, for any $\eps>0$,
\begin{eqnarray*}
\E\left(Z_{X_N}(I)^{[m]}\right)
&=&\int_{I^m\setminus D^{\,\eps}_m}\int_{\R^m}\prod_{i=1}^m |y_i| 
p_{N,{\textbf t}}(\textbf{0},\textbf{y})d\textbf{y}d\textbf{t}+O(\eps^{\frac{1}{5}}),
\end{eqnarray*}
where $p_{\textbf{t},N}(\textbf{0},\textbf{y})$ is the density of the vector $V_N({\textbf t})$ (see (\ref{vn}))
%$(X_N(t_1),\cdots,X_N(t_m),X_N'(t_1),\cdots,X_N'(t_m))$ 
evaluated at the point $(0,\ldots,0;y_1,\ldots,y_m)$, and $O(\cdot)$ is uniform with respect to $N$.
In order to achieve the proof we shall use a  {\it local} central limit theorem to guarantee the pointwise convergence of $p_{N,{\textbf t}}$ and then take the limit under the integral.

Denote by $\gamma_{{\bf t}}$ the density of the centered Gaussian vector with covariance $\Sigma({\bf t})$, where ${\bf t}$ and $\Sigma({\bf t})$ are like in Lemma \ref{lemma:non-deg}.
Thanks to Lemma \ref{lemma:llt}, the sequence of functions $(p_{N,{\textbf t}})_{N\geq N_0}$ 
is equicontinuous and equibounded. 
Moreover, since by the CLT one has that $V_N({\bf t})$ converges toward $N(0,\Sigma({\bf t}))$, the limit of any convergent subsequence
of $(p_{N,{\textbf t}})_{N\geq N_0}$ is necessarily $\gamma_{{\bf t}}$
by the Scheff\'e theorem.
By a compactness argument, it follows that $p_{N,{\textbf t}}$ converges uniformly on each compact of $\R^{2m}$ towards $\gamma_{{\bf t}}$. 
Besides, the bound \eqref{controleinfini} gives the suitable domination to pass the limit under the integral sign in the Rice's formula \eqref{eq:Rice2}. 
We obtain
\begin{equation}\label{presquefini}
\int_{I^m\setminus D^{\,\eps}_m}
\int_{\R^m}\prod_{i=1}^m |y_i| p_{N,{\textbf t}}(\textbf{0},\textbf{y})d\textbf{y}d\textbf{t}
\mathop{\to}\limits_{N\to\infty} 
\int_{I^m\setminus D^{\,\eps}_m}\int_{\R^m}\prod_{i=1}^m |y_i| \gamma_{{\bf t}}(\textbf{0},\textbf{y}) d\textbf{y} d\textbf{t}.
\end{equation}
It remains to let $\eps \to 0$ and we reach the desired result, namely
\begin{eqnarray*}
\lim_{N\to\infty}\E\left(Z_{X_N}(I)^{[m]}\right)&=&\int_{I^m}\int_{\R^m}\prod_{i=1}^m |y_i| \gamma_{{\bf t}}(\textbf{0},\textbf{y}) d\textbf{y} d\textbf{t}\\
&=&\E(Z_{W}(I)^{[m]}).
\end{eqnarray*}
\end{proof}

\section{Auxiliary results}

\subsection{Upper bounds for the density $p_{N,{\bf t}}$ and its gradient}

The next proposition provides useful bounds for the density $p_{N,{\bf t}}$ (as well as its gradient) of $V_N({\bf t})$, given by (\ref{vn}), in the case where $N$ is large enough and ${\bf t}$ does not belong to the $\eps$-enlargement of the set $D_m$ of hyper-diagonals.

\begin{lemma}\label{lemma:llt} 
Fix $\eps>0$. 
Then,  for any ${\bf t}\in I^m\setminus{D\eps_m} $ and any $N$ large enough (bigger or equal than $N_0$, say), the density $p _{N, \mathbf t}$ exists, is $\mathcal{C}^1$ and satisfies, for all $\mathbf{x},\mathbf{y}\in\R^{m}$,
\begin{eqnarray}\label{controleinfini}
\sup_{ \mathbf t \in I^m\setminus{D_\eps^m}}\,\,\sup_{N\geq N_0} \,\,
p _{N, \mathbf t}(\mathbf{x},\mathbf{y}) &\le& \frac{C_\eps }{(1+y_1^4)\cdots (1+y_{m}^4)}\leq C_\eps;\\
\label{lipschitz}
\sup_{ \mathbf t \in I^m\setminus{D_\eps^m}}\,\,\sup_{N\geq N_0} \,\,
\|\nabla p_{N, \mathbf t} \|_\infty &\le& C_\eps.
\end{eqnarray}
We recall that $C_\eps$ denotes a constant that only depends on $\eps$ and whose  value may change from one occurence to another.
\end{lemma}

\begin{remark}
{\rm
A careful inspection of the proof would show that the denominator in the right-hand side of (\ref{controleinfini}) may be replaced by any other polynomial in ${\bf x}$ and ${\bf y}$.
But because we will only need the one appearing in (\ref{controleinfini}), 
for simplicity we decided not to  state our Lemma \ref{lemma:llt} in such a general setting.
}
\end{remark}

Our proof of Lemma \ref{lemma:llt} will heavily rely on the use of the following
celebrated theorem due to Paul Malliavin \cite{Mallia}.
\begin{theorem}[Malliavin]\label{paul}
Let $\mu$ be a finite signed measure over $\R^{2m}$.
\begin{itemize}
\item[(i)]If, for all $i=1,\ldots,2m$ there exists a constant $C_i$ satisfying that for any $\phi \in \mathcal{C}^\infty_c(\R^{2m},\R)$, 
\begin{equation}\label{criteredensite}
\left|\int_{\R^{2m}} \partial_i \phi (x) d\mu(x)\right|\leq C_i \|\phi\|_\infty,
\end{equation}
then $\mu$ is absolutely continuous with respect to Lebesgue measure of $\R^{2m}$.
\item[(ii)] If, for any  multi-index $\alpha$ there exists a constant $C_\alpha>0$ satisfying that for any $\phi \in \mathcal{C}^\infty_c(\R^{2m},\R)$,
\begin{equation}\label{regudensi}
\left|\int_{\R^{2m}} \partial_\alpha \phi (x) d\mu(x)\right|\leq C_\alpha \|\phi\|_\infty,
\end{equation}
then $\mu$ admits a density in the class $\mathcal{C}^\infty_b(\R^{2m},\R)$ of $C^\infty$ functions which are bounded together with all its derivatives.
\end{itemize}
\end{theorem}
Actually, we shall rather use the following criterion which is an immediate consequence of Theorem \ref{paul}.
\begin{corollary}\label{criteresequentiel}
Consider a sequence of finite signed measures $\mu_N$ over $\R^{2m}$ such that, for any multi-index $\alpha$, we may find a constant $C_{\alpha}>0$ satisfying that for any $\phi \in \mathcal{C}^\infty_c(\R^{2m},\R)$,
\begin{equation}\label{criteresequentieleq}
\sup_N \left|\int_{\R^{2m}} \partial_\alpha \phi (x) d{\mu_N}(x)\right|\leq C_\alpha \|\phi\|_\infty.
\end{equation}
Then, the sequence of densities of $\mu_N$ 
is uniformly bounded (by a constant only depending on $C_\alpha$) for the (nuclear) topology of $\mathcal{C}^\infty_b(\R^{2m},\R)$.
\end{corollary}

We are now in a position to prove Lemma \ref{lemma:llt}.
First, let us introduce the formalism of integrations by parts. % we shall use to derive the aforementioned local criterion. 
For any pair $(\Psi_1,\Psi_2)$ of $\mathcal{C}^1(\R^{2m},\R)$, let us set
\begin{equation}\label{Gamma}
\Gamma[\Psi_1(\textbf{a},\textbf{b}),\Psi_2(\textbf{a},\textbf{b})]=\sum_{i=1}^{2m}\partial_i \Psi_1(\textbf{a},\textbf{b})\partial_i \Psi_2(\textbf{a},\textbf{b}).
\end{equation}
Also, for $F\in \mathcal{C}^2(\R^{2m},\R)$, set
\begin{equation}\label{Generateur}
L[F(\textbf{a},\textbf{b})]=\sum_{i=1}^{2m} \partial^2_{i,i} F(\textbf{a},\textbf{b})+\sum_{i=1}^{m} \partial_i F(\textbf{a},\textbf{b})\frac{\rho'(a_i)}{\rho(a_i)}+\sum_{i=m+1}^{2m}\partial_i F(\textbf{a},\textbf{b})\frac{\rho'(b_i)}{\rho(b_i)}.
\end{equation}

In order to simplify the notation in (\ref{Gamma})-(\ref{Generateur}) let us use here and throughout the text the shorthand notation $\textbf{a}=(a_1,\cdots,a_m)$ and $\textbf{b}=(b_1,\cdots,b_m)$. Besides, in the sequel $(\textbf{a},\textbf{b})$ will denote the $2m$-uple $(x_1,\cdots,x_{2m})$ such that $(x_1,\cdots,x_m)=\textbf{a}$ and $(x_{m+1},\cdots,x_{2m})=\textbf{b}$. The key relationship between the operators $L$ and $\Gamma$ is the following integration by parts formula: for any $\Psi_1 \in\mathcal{C}^1_b(\R^{2m},\R)$ and $\Psi_2 \in \mathcal{C}^2_b(\R^{2m},\R)$, it holds that
\begin{eqnarray}\notag
&&\int_{\R^{2m}}\Gamma\left[\Psi_1(\textbf{a},\textbf{b}),\Psi_2(\textbf{a},\textbf{b})\right]\rho(a_1)\rho(b_1)\ldots
\rho(a_m)\rho(b_m)d\textbf{a}d\textbf{b}\\
&=&-\int_{\R^{2m}}\Psi_1(\textbf{a},\textbf{b}) L\left[\Psi_2(\text{a},\textbf{b})\right]\rho(a_1)\rho(b_1)\ldots
\rho(a_m)\rho(b_m)d\textbf{a}d\textbf{b}.\label{IPP}
\end{eqnarray}

Let us apply this formalism to our problem.  
Fix ${\bf t}=(t_1,\cdots,t_m)\in I^m\setminus D^{\,\eps}_m$.
We have, for any $i\neq j$,
\begin{eqnarray*}
\Gamma[X_N(t_i),X_N(t_i)]&=& 1,\\
\Gamma[X_N'(t_i),X_N(t_i)]&=& 0,\\
\Gamma[X_N'(t_i),X_N'(t_i)]&=& \frac{1}{N}\sum_{n=1}^N \frac{n^2}{N^2},\\
\Gamma[X_N(t_i),X_N(t_j)]&=& \frac{1}{N}\sum_{n=1}^N \cos\left(\frac{ n (t_i-t_j)}{N}\right),\\
\Gamma[X_N(t_i),X_N'(t_j)]&=& \frac{1}{N}\sum_{n=1}^N \frac{n}{N}\sin\left(\frac{ n (t_i-t_j)}{N}\right),\\
\Gamma[X_N'(t_i),X_N'(t_j)]&=& -\frac{1}{N}\sum_{n=1}^N \frac{n^2}{N^2}\cos\left(\frac{ n (t_i-t_j)}{N}\right).
\end{eqnarray*}
Recall from (\ref{vn}) that $V_N({\bf t})=(X_N(t_1),\cdots,X_N(t_m);X_N'(t_1),\cdots,X_N'(t_m))$. Let us also denote 
by $\widehat{\Gamma}_N$ the  \textit{Malliavin matrix} associated with $V_N$:

$$\widehat{\Gamma}_N({\bf t})=\big(\Gamma[ V_N({\bf t})_i, V_N({\bf t})_j]\big)_{1\le i,j \le 2m}.$$

\begin{remark}
{\rm
At this stage, it is worthwhile noting that many other choices for the operators $\Gamma$ and $L$ could have led to integration by parts formulas. However the choice we made here seems to be the only reasonable one leading to a {\it deterministic} (that is, independent of $a_n,b_n$) Malliavin matrix. As we will see, this determinacy will turn out to be very helpful and will play an important role in our reasoning.
}
\end{remark}

Consider $r_N(x) = \frac{1}{N}\sum_{n=1}^N \cos ( \frac{n}{N} x )$
and recall that $\sinc(x)=\sin x/x$.
It is proved in \cite[Section 3.1]{granville} 
that $r_N,r_N',r_N''$ converge toward $\sinc,\sinc',\sinc''$ 
uniformly on every compact as $N\to\infty$. 
As a result, the matrix $\widehat{\Gamma}_N({\bf t})$ converge uniformly over $I^m$ towards the matrix 
$
\Sigma({\bf t})$ of Lemma \ref{lemma:non-deg}.
Still by  Lemma \ref{lemma:non-deg}, the
determinant of $\Sigma({\bf t})$ is non-zero on 
$I^m\setminus D_m$.
Fix $\eps>0$. By a classical compactness argument, we deduce from the previous fact that 
their exist $N_0=N_0(\eps)\in\mathbb{N}$ and $\eta=\eta(\eps)>0$ such that 
$$
\forall N\geq N_0,\,\,\forall {\bf t}\in I^m \setminus D_m^{\,\eps}:\quad
\det(\widehat{\Gamma}_N({\bf t}))\geq \eta.
$$

The following class of random variables will naturally be present in the weights that will appear after applying repeatedly the integration by parts (\ref{IPP}). Set 
\begin{eqnarray*}
\mathcal{A}_0(N,{\bf t})&=&\left\{X_N(t_1),\cdots,X_N(t_m),X_N'(t_1),\cdots,X_N'(t_m)\right\}.\\
\mathcal{A}_1(N,{\bf t}) &=&L (\mathcal{A}_0(N,{\bf t})) \cup \Gamma[\mathcal{A}_0(N,{\bf t}),\mathcal{A}_0(N,{\bf t})]\\
&\vdots&\\
\mathcal{A}_{r+1}(N,{\bf t})&=&L(\mathcal{A}_r(N,{\bf t})) \cup \Gamma[\mathcal{A}_r(N,{\bf t}),\mathcal{A}_r(N,{\bf t})]\\
&\vdots&\\
\mathcal{A}_\infty (N,{\bf t})&=& \bigcup_{r=0}^\infty \mathcal{A}_r(N,{\bf t}),
\end{eqnarray*}
where, for a given set $\mathcal{E}=\{e_1,\cdots,e_s\}$, $L(\mathcal{E})$ denotes $\{ L[e_1],\cdots,L[e_s]\}$ while $\Gamma[\mathcal{E},\mathcal{E}]$ stands for $\{\Gamma[e_i,e_j]\,|\,1\le i,j\le s\}$.

To achieve the proof of Lemma \ref{lemma:llt}, we will need the following result.
 
\begin{lemma}\label{superborne}
Suppose that Assumption \ref{Hypo-densite}
is in order.
Then, for any $s\geq 2$ there exists a constant $C_s>0$ such that,
for any $N$, any ${\bf t}\in I^m$ and any $U_N({\bf t})\in\mathcal{A}_\infty (N,{\bf t})$,
\begin{equation}\label{megabornitude}
\E[|U_N({\bf t})|^s]\leq C_s.
\end{equation}
\end{lemma}

\begin{proof} \underline{Step 1}: {\it explicit description of the elements of $\mathcal{A}_\infty(N,{\bf t})$}.
Let us establish by induction that, for any $r\ge 0$, the elements of $\mathcal{A}_r(N,{\bf t})$ are of the form
\begin{equation}\label{i}
\frac{1}{N^p}\sum_{n=1}^N \frac{n^q}{N^q} \left(   \alpha f(a_n) \phi_{n,N}(\textbf{t}) + \beta f (b_n) \psi_{n,N}(\textbf{t})\right),
\end{equation}
with $\alpha,\beta \in \{-1,1\}$, $p\in\{\frac12\}\cup[1,\infty)$, $q\in \N$, $f$ some $\mathcal{C}^\infty$ function such that  $\displaystyle{f^{(l)}(a_1)\in \bigcap_{s>1} L^s(\Omega)}$ for any $l\ge 0$, and where $\phi_{n,N}(\textbf{t}),\psi_{n,N}(\textbf{t})$ are products of at most $2^{2^r}$ terms among $\cos(\frac{n t_i}{N}),\sin(\frac{n t_i}{N})$, $1\le i\le m$. Moreover, when $p=\frac12$ in (\ref{i}), we have that $\E[f(a_1)]=0$.

It is immediate that the elements of $\mathcal{A}_0(N,{\bf t})$ are of the form (\ref{i}). Assume now that the above description of $\mathcal{A}_r(N,{\bf t})$ has been established until the rank $r \ge 0$. Applying $L$ to some element of $\mathcal{A}_r(N,{\bf t})$ of the form (\ref{i}) leads to

\begin{eqnarray*}
&&\frac{1}{N^p}\sum_{n=1}^N \frac{n^q}{N^q} \left(\alpha L[f(a_n)] \phi_{n,N}(\textbf{t})+ \beta L[f(b_n)] \psi_{n,N}(\textbf{t})\right)\\
&=&\frac{1}{N^p}\sum_{n=1}^N \frac{n^q}{N^q} \left(\alpha \left(f''(a_n)+f'(a_n)\frac{\rho'(a_n)}{\rho(a_n)}\right) \phi_{n,N}(\textbf{t}) \right.\\
&&\left.\hskip2.4cm+ \beta \left(f''(b_n)+f'(b_n)\frac{\rho'(b_n)}{\rho(b_n)}\right) \psi_{n,N}(\textbf{t})\right),
\end{eqnarray*}
which is again of the form (\ref{i}). Indeed, $\E(L[f(a_1)])=-\E(\Gamma[1,f(a_1)])=0$ and by our assumptions on $f$ and $\frac{\rho'}{\rho}$ it holds that $g:=f''+f' \frac{\rho'}{\rho}$ is such that $\displaystyle{g^{(l)}(a_1)\in \bigcap_{s>1} L^s(\Omega)}$ for any $l\ge 0$. Now, applying the bilinear form $\Gamma$ to two elements of $\mathcal{A}_r(N,{\bf t})$, say
\begin{eqnarray*}
U_N(\textbf{t})&=&\frac{1}{N^p}\sum_{n=1}^N \frac{n^q}{N^q} \left(   \alpha f(a_n) \phi_{n,N}(\textbf{t})+ \beta f (b_n) \psi_{n,N}(\textbf{t})\right)\\
V_N(\textbf{t})&=&\frac{1}{N^{p'}}\sum_{n=1}^N \frac{n^{q'}}{N^{q'}} \left( \alpha' g(a_n) \widetilde{\phi_{n,N}}(\textbf{t}) + \beta' g(b_n) \widetilde{\psi_{n,N}}(\textbf{t})\right),\\
\end{eqnarray*}
leads us to
\begin{eqnarray*}
\Gamma[U_N(\textbf{t}),V_N(\textbf{t})]&=&
\frac{1}{N^{p+p'}}\sum_{n=1}^N \frac{n^{q+q'}}{N^{q+q'}}\left(\alpha \alpha' f'(a_n)g'(a_n)\phi_{n,N}(\textbf{t})\widetilde{\phi_{n,N}}(\textbf{t})\right.\\
&&\left.\hskip3cm+\beta \beta' f'(b_n)g'(b_n)\psi_{n,N}(\textbf{t})\widetilde{\psi_{n,N}}(\textbf{t})\right).
\end{eqnarray*}
We easily observe that $f' g'$ together with all its derivatives are in $\bigcap_{s>1} L^s(\Omega)$, that $\alpha\alpha', \beta\beta' \in \{-1,1\}$ and that $\phi_{n,N}(\textbf{t})\widetilde{\phi_{n,N}}(\textbf{t}),\psi_{n,N}(\textbf{t})\widetilde{\psi_{n,N}}(\textbf{t})$ both contain at most $2^{2^{r+1}}$ terms. 
~\\\\
\underline{Step 2}: {\it bounding the elements of $\mathcal{A}_\infty(N,{\bf t})$}. Fix $s\geq 2$, and let us consider the following element in $\mathcal{A}_\infty(N,{\bf t})$ of the form (\ref{i}):
$$U_N(\textbf{t})=\frac{1}{N^p}\sum_{n=1}^N \frac{n^q}{N^q} \left(   \alpha f(a_n) \phi_{n,N}(\textbf{t})+ \beta f (b_n) \psi_{n,N}(\textbf{t})\right).$$

Relying on Step 1, we may infer that $\sup_{\textbf{t}\in [0,T]^m} \left|\phi_{k,N}(\textbf{t})\right|+\left|\psi_{k,N}(\textbf{t})\right|\le 2$. As a result, when $p\geq 1$ and using the triangle inequality for the norm $\|\cdot\|_s$, one can write
\begin{eqnarray*}
\left\|U_N(\textbf{t})\right\|_s\le \left(\frac{4}{N^p}\sum_{k=1}^N\frac{k^q}{N^q}\right)\left\|f(a_1)\right\|_s
\le\frac{4}{N^{p-1}}\le 4\left\|f(a_1)\right\|_s.
\end{eqnarray*}
Let us now consider the situation where $p=\frac12$ and recall that $\E[f(a_1)]=0$ in this case, implying in turn that $E[U_N({\bf t})]=0$.
Due to this latter property, the following sequence is a martingale:
$$
M_N=\sum_{k=1}^N k^q\left\{\alpha f(a_k) \phi_{k,N}(\textbf{t})+ \beta f(b_k) \psi_{k,N}(\textbf{t})\right\}.$$
The $s/2$-moment of its quadratic variation can be bounded as follows:
\begin{eqnarray*}
&&\E\big[\langle M_N, M_N\rangle^{s/2}\big] \\
&=& 
\E\left[\left(
\sum_{k=1}^N k^{2q}\left\{f^2(a_k) \phi_{k,N}^2(\textbf{t})+ f^2(b_k) \psi_{k,N}^2(\textbf{t})\right\}\right)^{s/2}
\right]
\\
&\le& 2N^{(q+\frac12)s}\,\,\E[|f(a_1)|^s].
\end{eqnarray*}
Applying Burkholder-Davis-Gundy inequality to $M_N$ leads to
\begin{eqnarray*}
\E\left(|M_N|^s\right)&\le& C_s \,\E\left(\langle M_N, M_N\rangle^{\frac s 2}\right)=\mathcal{O}\left(N^{(q+\frac12)s}\right).
\end{eqnarray*}
Finally,
\begin{eqnarray*}
\E\left(\left|U_N(\textbf{t})\right|^s\right)=\E\left(\left|\frac{M_N}{N^{q+\frac{1}{2}}}\right|^s\right)=\mathcal{O}\left(1\right),
\end{eqnarray*}
where the last bound is uniform in $\textbf{t}\in I^m$.
This concludes the proof in the case $p=\frac12$ as well.
\end{proof}

Now Proposition \ref{superborne} has been established, let us do the proof of Lemma  \ref{lemma:llt}.

\begin{proof}
For any $\Phi:\R^{2m}\to \R$ the chain rule for $\Gamma$ leads to
\begin{eqnarray*}
\Gamma[\Phi(V_N({\bf t})),V_N({\bf t})_j]&=&\sum_{i=1}^{2m} \partial_i \Phi (V_N({\bf t})) \Gamma[V_N({\bf t})_i,V_N({\bf t})_j].\\
\end{eqnarray*}
As a result, setting $W_N({\bf t})$ to be the vector $(\Gamma[\Phi(V_N({\bf t})),V_N({\bf t})_j])_{1\le j \le 2m}$, the previous equation can be written as
$$
W_N({\bf t}) = \widehat{\Gamma}_N({\bf t})\times \nabla \Phi ( V_N ({\bf t})).
$$
Recalling that $\widehat{\Gamma}_N({\bf t})$ is invertible on $I^m \setminus D_m^{\,\eps}$, it follows that 
\begin{equation}\label{inverse}
\nabla \Phi(V_N({\bf t})) = \widehat{\Gamma}_N({\bf t})^{-1} \times W_N({\bf t}).
\end{equation}
Fix $\eps>0$, ${\bf t}\in I^m\setminus D_m^\eps$, 
a multi-index $\alpha$, a polynomial $Q:\R^{2m}\to \R$ and a test function $\Psi$. One has, with $\theta=\partial_{\alpha\setminus\{\alpha_1\}} \Psi,$
\begin{eqnarray*}
&&\E\left(Q(V_N({\bf t}))\partial_\alpha\Psi(V_N({\bf t}))\right)
=\E\left(Q(V_N({\bf t}))
\partial_{\alpha_1} \theta(V_N({\bf t}))\right)\\
&=&\sum_{j=1}^{2m}(\widehat{\Gamma}_N({\bf t})^{-1})_{\alpha_1,j}\,\E\left(Q(V_N({\bf t}))\Gamma[\theta(V_N({\bf t})),V_N({\bf t})_j]\right)\\
&=&-\sum_{j=1}^{2m}(\widehat{\Gamma}_N({\bf t})^{-1})_{\alpha_1,j}\,\E\left(Q(V_N({\bf t}))L[V_N({\bf t})_j]\theta(V_N({\bf t}))\right)\\
&&-\sum_{j=1}^{2m}(\widehat{\Gamma}_N({\bf t})^{-1})_{\alpha_1,j}\,\E\left(\Gamma[Q(V_N({\bf t})),V_N({\bf t})_j]\theta(V_N({\bf t}))\right)\\
&=&\\
&\vdots&\\
&=& \E(\Psi(V_N({\bf t})) H_N({\bf t}) ),
\end{eqnarray*}
where $H_N({\bf t})$ is an element of the algebra generated by $\mathcal{A}_\infty(N,{\bf t})$. 
Here, we applied \eqref{inverse} in the second equality and \eqref{IPP} 
in the third equality; also, we used
routine calculus to deal with the term of the form $\E(\Psi_1\Gamma(\Psi_2,\Psi_3))$. 

By virtue of Lemma \ref{superborne}, $\sup_N\sup_{t\in I^m}\E\big[|H_N({\bf t})|^q\big]<\infty$ for any $q$. Finally, choosing $Q({\bf x},{\bf y})=(1+y_1)^4\ldots(1+ y_m)^4$ yields
$$\sup_{N\ge N_0} \left|\int \partial_\alpha \Psi(x)d\mu_N(x)\right|\leq C_\alpha \|\Psi\|_\infty$$
with $\mu_N$ the measure with density 
$(1+y_1^4)\ldots (1+y_m^4)p_{N,{\textbf t}}(\textbf{x},\textbf{y}) d \textbf{x} d\textbf{y}.$ 
In virtue of (\ref{criteresequentieleq}), we may infer that (\ref{controleinfini}) takes place. 
On the other hand, using the same criterion with $Q=1$ leads this time to
(\ref{lipschitz}). 

\end{proof}

\subsection{Rice's Formulas}
Rice's formulas are integral formulas 
for the (factorial) moments of the number of crossings of a 
stochastic process within a given interval.  They are true for Gaussian processes under minimal hypotheses.
However,  since we are here dealing with {\it non}-Gaussian processes, we have to be careful and to check their validity.
General results allowing to do so exist in the literature (see, e.g., 
Theorem 3.4  in \cite{aw} or Theorem 11.2.1 in \cite{adler}) but they rely 
on rather heavy conditions. This is why we prefer here to give a simple proof for smooth processes, which is well suited for our need.

\begin{proposition}[Rice's Formula for smooth processes] \label{prop:Rice}
Let  $m$ be a positive integer, let $I$ be a compact interval  and let $Y$ be a process satisfying that:
\begin{enumerate}
\item[A1.] it has $ \mathcal{C}^1$  sample paths;
 \item[A2.]  the one-dimensional  density $y\mapsto p_{Y(t)}(y) $ is uniformly bounded for  $t \in I$ and for $y$ in a neighborhood of a certain level  $u$;
 \item[A3.]  the number $Z_{\dot Y}(I)$ of zeros of the derivative $\dot Y$ of $Y$ within $I$ admits a moment of order $m$;
 \item[A4.] for any  pairwise disjoint intervals $J_1, \ldots, J_m$ included in $I$,
the Rice function
 \begin{eqnarray*}
 && B_{J_1\times\cdots \times J_m} (v_1,\ldots,v_m) \\
 &=& \int_{J_1\times\cdots \times J_m}  \E(|Y'(t_1)|\ldots |Y'(t_m)|\big| Y(t_1)=v_1,\ldots Y(t_m)=v_m)\\
  &&\hskip4.2cm \times p_{Y(t_1),\ldots, Y(t_m)} (v_1,\ldots,v_m)  dt_1\ldots dt_m
  \end{eqnarray*}
  is  well defined and continuous at $(u,\ldots,u)$.
\end{enumerate}
Then $Y$ satisfies  the $m^{\mbox{th}}$ Rice's formula, that is, 
\begin{enumerate}
\item[(i)]
for any  pairwise disjoint intervals $J_1, \ldots, J_m$ included in $I$,
$$
 \E (Z^u_Y(J_1)  \times \cdots \times Z^u_Y(J_m))  =  B_{J_1\times\cdots \times J_m} (u, \ldots,u);
 $$
\item[(ii)]
\begin{equation*}
\E(Z^u_Y(I)^{[m]} )=B_{I^m}(u, \ldots u),
\end{equation*}
where $Z^u_Y(I)$ denotes the number of crossing of the level $u$  on the interval $I$.
\end{enumerate}
\end{proposition}

\begin{corollary}\label{cor:Rice-trunc}
For $\eps>0$, we have
\begin{equation*}
\E(Z_{X_N}(I) ^{[m]})
=\int_{I^m\setminus D^{\eps}_m}
\int_{\R^m}|y_1|\ldots |y_m | \,p_{N,{\mathbf  t}}
({\mathbf  0};{\mathbf  y})
d{\mathbf  y} d{\mathbf  t}+ O(\eps^{1/5}).
\end{equation*}
The constants involved in the  Landau notation 
depend on $m$ and $I$ but \underline{not} on $N$.
\end{corollary}

Let us first prove Proposition \ref{prop:Rice}.
\begin{proof}[Proof of Proposition \ref{prop:Rice}]

 We begin with the case $m=1$. By assumption, 
the process  $Y$ has $ \mathcal{C}^1$ sample paths and $Y_t$  admits  a uniformly bounded density. Ylvisaker theorem (see, e.g., \cite[Theorem 1.21]{aw}) implies that almost surely  there is no point $t \in I$ such that $Y(t) =0$ and $Y'(t)=0$. As a consequence,  the number of crossings of the level $u$  is almost surely finite  and  we can apply the Kac formula (Lemma 3.1 in \cite{aw}), according to which 
\begin{equation} \label{e:kac}
Z^u_Y(I) =\lim_{\delta \to 0} Z^u_\delta(I), 
\end{equation}
where  
$$
Z^u_\delta(I) :=  \frac {1} {2\delta} \int _0^T \indicator_{\{|Y(t) -u|\leq \delta\}}  |Y'(t)| dt.
$$
It is easy to check that 
$$
Z^u_\delta(I) \leq Z_{\dot Y}(I) +1 .
$$
By dominated convergence in \eqref{e:kac}, we get that 
$$
\E [Z^u_Y(I)] =     \lim_{\delta \to 0} \frac {1} {2\delta} \int_{u-\delta} ^{u+\delta} B_{[0,T]}(v) dv = B_{[0,T]}(u).
$$
The last equality comes from the continuity of $B_{[0,T]}$ at $u$. 

We turn now to the case $m>1$. 
Let $C_u(I) $  denote the set of those $t\in I$ such that $Y(t) =u$. 
Since the set $D_m$ of hyperdiagonals  of $I^m$
has  Lebesgue measure zero and since 
 $$
 Z^u_Y (I)^{[m]} = {\rm Card}\big(C_u(I)\setminus D_m\big),
 $$
it is sufficient to prove $(i)$, that is, for pairwise disjoint  interval $J_1\times \cdots J_m$,
\begin{equation}\label{e:rice} 
\E \big( Z^u_Y(J_1) \times \cdots \times Z^u_Y(J_m) \big) 
=B_{J_1\times\ldots\times J_m}(u,\ldots,u).
  \end{equation}
The result $(ii)$ for $I^m$ will then follow from a standard approximation argument using the absolute continuity  of the measure defined by the right-hand-side. 

To prove \eqref{e:rice},  we use Kac's formula and dominated convergence, exactly as in  the case $m=1$.
\end{proof}

Finally, let us do the proof of Corollary \ref{cor:Rice-trunc}.
It will rely on several lemmas, that may have their own interests.

\begin{lemma}\label{lm1}
Fix a compact interval $I$ of length $|I|$. For any $N$ and  $k$,
\begin{equation}\label{81}
\E\left[\sup_{t\in I}|X_N^{(k)}(t)|^2\right]\leq  2(1+|I|^2).
\end{equation}
\end{lemma}
\begin{proof}
Assume that $I=[a,b]$. The proof is divided into two steps.\\

{\it First step}. If $f:[a,b]\to\R$ is a $\mathcal{C}^1$ function, then we can
straightforwardly check that, for $t\in [a,b]$,
\begin{equation*}
f(t)=\frac{1}{b-a}\int_a^{b}f(s)ds + \frac{1}{b-a}\int_a^{t}(s-a)f'(s)ds
+\frac{1}{b-a}\int_t^{b}(s-b)f'(s)ds.
\end{equation*}
As a result,
\begin{equation*}
\sup_{t\in I}|f(t)|\leq \frac{1}{b-a}\int_a^{b}|f(s)|ds + \int_a^{b}|f'(s)|ds,
\end{equation*}
implying in turn, due to $(x+y)^2\leq 2x^2+2y^2$ and Cauchy-Schwarz inequality,
\begin{equation}\label{ine}
\sup_{t\in I}|f(t)|^2\leq \frac2{b-a}\int_a^{b}f(s)^2ds + 2(b-a)\int_a^{b}f'(s)^2ds.
\end{equation}

\bigskip

{\it Second step}. For any $N$ and any $l$,
\begin{eqnarray*}
X_N^{(2l)}(t) &=& \frac{1}{\sqrt{N}}\sum_{j=1}^N a_j\,(-1)^l \left(\frac{j}{N}\right)^{2l}
\cos(\frac{j}{N}t) + b_j\,(-1)^l \left(\frac{j}{N}\right)^{2l}
\sin(\frac{j}{N}t) ,\\
X_N^{(2l+1)}(t) &=& \frac{1}{\sqrt{N}}\sum_{j=1}^N a_j\,(-1)^{l+1} \left(\frac{j}{N}\right)^{2l+1}
\sin(\frac{j}{N}t) + b_j\,(-1)^l \left(\frac{j}{N}\right)^{2l+1}
\cos(\frac{j}{N}t).
\end{eqnarray*}
As a consequence, for any $N$ and any $k$,
\[
\E[X_N^{(k)}(t)^2] = \frac{1}{N}\sum_{j=1}^N \left(\frac{j}{N}\right)^{2k} \leq 1.
\]
The conclusion (\ref{81}) thus follows by plugging the previous inequality into (\ref{ine}).
\end{proof}

\begin{lemma}\label{lm3}
For any interval $I$  of length $|I|$ and any integers $k\geq 1$ and $N\geq 1$,
\begin{equation}\label{ineq1}
\Prob(Z_{X_N}(I)\geq k)\leq
(Const)
(k!(k-1)!)^{-\frac12}\,|I|^{k-\frac1{2}}.
\end{equation}
In particular, for any $r>1$ and any interval $I$, 
we have the following uniform bound:
\begin{equation}\label{ineq2}
\E[Z_{X_N}(I)^r] \leq  (Const) \left(\sum_{k=1}^\infty \frac{|I|^{\frac{2k-1}{2r}}}{(k!(k-1)!)^{\frac1{2r}}}\right)^r.
\end{equation}
\end{lemma}
\begin{proof}
Let us first concentrate on the inequality \eqref{ineq1}. Throughout its proof, we will need the following result 
which follows from
Lagrange formula  for the difference between  the function and  its polynomial interpolation (which vanishes), see, e.g., \cite{davis}. \\

{\bf Claim}: Assume that $f:[a,b]\to\R$ is of class $\mathcal{C}^k$ ($k\geq 1$) and that their exist $x_1,\ldots,x_k\in[a,b]$ (possibly repeated) such that $f(x_1)=\ldots=f(x_k)=0$. Then their exist $y_1,\ldots,y_{k-1}\in[a,b]$ (possibly repeated) such that $f'(y_1)=\ldots=f'(y_{k-1})=0$; moreover, for all $x\in[a,b]$ there exist $\xi,\eta\in(a,b)$ such that
\begin{equation}\label{rolle}
f(x)=\frac{1}{k!}f^{(k)}(\xi)\prod_{i=1}^k(x-x_i)
\quad
\mbox{and}\quad
f'(x)=\frac{1}{(k-1)!}f^{(k)}(\eta)\prod_{i=1}^{k-1}(x-y_i).
\end{equation}
\bigskip 

Thanks to the conclusion of the previous claim, we can now decompose our probability of interest in a clever way, by introducing an extra parameter $M>0$ whose value will be optimized in the end.
From (\ref{rolle}), one easily deduces that, if $\sup_{t\in I}|X_N^{(k)}(t)|\leq M$, then
$|X_N(c)|\leq \frac{M |I|^k}{k!}$ and $|X'_N(c)|\leq \frac{M |I|^{k-1}}{(k-1)!}$, where $c$ denote the middle point of the interval $I$ (say). We thus have
\begin{eqnarray*}
&&\Prob (Z_N(I)\geq k)\\
&\leq& \Prob\left(\sup_{t\in I}|X_N^{(k)}(t)|> M\right)+ \Prob\left(Z_N(I)\geq k,\,\sup_{t\in I}|X_N^{(k)}(t)|\leq M\right) \\
&\leq&\frac1{M^2}\,\E\left[\sup_{t\in I}|X_N^{(k)}(t)|^2\right]
+ \Prob\left(
|X_N(c)|\leq \frac{M |I|^k}{k!},\,|X'_N(c)|\leq \frac{M |I|^{k-1}}{(k-1)!}
\right)\\
&\leq& (Const) \Big(
\frac{1}{M^2}
+ \frac{M^2|I|^{2k-1}}{k!(k-1)!}\Big)\quad\mbox{(by Lemmas \ref{lm1} and \ref{lemma:llt})}.
\end{eqnarray*}
Choosing  $M^2=\sqrt{k!(k-1)!}\,|I|^{\frac{1-2k}{2}}$ leads to the desired conclusion (\ref{ineq1}).

Now, let us focus on (\ref{ineq2}). Using Fubini and then H\"older with $a=\frac{r}{r-1}$ and $b=r$, we can write
\begin{eqnarray*}
\E[Z_N(I)^r] &=& \sum_{k=1}^\infty \E [Z_N(I)^{r-1}{\bf 1}_{\{Z_N(I)\geq k\}}]\leq \E[Z_N(I)^r]^{\frac{r-1}{r}} \sum_{k=1}^\infty \Prob(Z_N(I)\geq k)^{\frac1r},
\end{eqnarray*}
so that, using (\ref{ineq1}),
\[
\E[Z_N(I)^r]\leq \left(\sum_{k=1}^\infty \Prob(Z_N(I)\geq k)^{\frac1r}\right)^r
\leq (Const) \left(\sum_{k=1}^\infty \frac{|I|^{\frac{2k-1}{2r}}}{(k!(k-1)!)^{\frac1{2r}}}\right)^r,
\]
which is exactly (\ref{ineq2}).
\end{proof}

We are now in a position to prove Corollary \ref{cor:Rice-trunc}.\\

{\it Proof of Corollary \ref{cor:Rice-trunc}}. 
First, let us check that the assumptions A1 to A4 of Proposition \ref{prop:Rice} are satisfied for $u=0$ and 
$Y=X_N$: A1 is obvious; A2 follows from Lemma \ref{lemma:llt}; A3 holds since the number of zeros of any trigonometric polynomial is bounded by two times its degree; and finally  A4 is an immediate consequence of (\ref{controleinfini}) in Lemma \ref{lemma:llt}.
We deduce that
$$
\E(Z_{X_N}(I) ^{[m]})
=\int_{I^m}
\int_{\R^m}|y_1|\ldots |y_m | \,p_{N,{\mathbf  t}}
({\mathbf  0};{\mathbf  y})
d{\mathbf  y} d{\mathbf  t}.
$$
To conclude, we are thus left to show that
$$
 \int_{D^\eps_m}
\int_{\R^m}|y_1|\ldots |y_m | \,p_{N,{\mathbf  t}}
({\mathbf  0};{\mathbf  y})
d{\mathbf  y} d{\mathbf  t}=O(\eps^{\frac{1}{5}}).
$$
To do so, consider the measure
$\mu_N$ defined on $\mathcal{B}(I^m)$ by
\[
\mu_N(J)= 
\E\big[{\rm Card}(J\cap C_N(I)^m)\big],
\]
where $C_N(I)$ is the set of zeros of $X_N$ lying in $I$. 
We know from Proposition \ref{prop:Rice} that $\mu_N$ 
restricted to $I^m\setminus D_m$ is absolutely continuous with respect  to Lebesgue measure and that $$\mu_N(D^m_\eps)=
 \int_{D^\eps_m}
\int_{\R^m}|y_1|\ldots |y_m | \,p_{N,{\mathbf  t}}
({\mathbf  0};{\mathbf  y})
d{\mathbf  y} d{\mathbf  t}.$$
Moreover, it is easy to check that 
\[
\mu_N(I^m) = \mu_N(I^m\setminus D_m) =
 \E [    {\rm Card}\big( C_N(I) ^m  \setminus D_m\big)]  = \E \big( Z_N(I)^{[m]}\big),
\]
and that 
for any (non necessarily disjoint) intervals $J_1,\ldots,J_k\subset I$
and any sequence $r_1,\ldots,r_k\geq 1$ of integers satisfying $r_1+\ldots+r_k=m$,
\[
\mu_N(J_1^{r_1}\times\ldots\times J_k^{r_k}) = \E[  {\rm Card}\big( C_N(J_1) ^{r_1}\times\ldots\times C_N(J_k)^{r_k}  \setminus D_m \big)].
\] 
It is also clear that
\begin{eqnarray*}
&&{\rm Card}\big( C_N(J_1) ^{r_1}\times\ldots\times C_N(J_k)^{r_k}  \setminus D_m \big)\\
&\le&
{\rm Card}\big(
C_N(J_1) ^{r_1} \setminus D_{r_1}\times \ldots\times
C_N(J_k) ^{r_k} \setminus D_{r_k}
\big) = Z_N(J_1)^{[r_1]}\ldots Z_N(J_k)^{[r_k]};
 \end{eqnarray*}
as a consequence 
\[
\mu_N(J_1^{r_1}\times\ldots\times J_k^{r_k})\leq \E ( Z_N(J_1)^{[r_1]}\ldots Z_N(J_k)^{[r_k]}).
\]

With all these properties at hand, we are now ready to conclude the proof of Corollary \ref{cor:Rice-trunc}, by showing that
\begin{equation}\label{cequilreste}
\sup_{N \geq N_0} \mu_N(D_m^{\,\eps})  = O( \eps^{\frac15})\quad\mbox{as $\eps\to 0$}.
\end{equation}
Firstly, we observe that
\[
D_m^{\,\eps} = \bigcup_{1\leq i\neq j\leq m} \Delta_{\eps,i,j},
\]
where $\Delta_{\eps,i,j}=\{(t_1,\ldots,t_m)\in I^m:\,|t_i-t_j|\leq\eps\}.$
Thus, to prove (\ref{cequilreste}) we are left 
to show that $\sup_{N\geq N_0} \E[\mu_N(\Delta_{\eps,i,j})]= O( \eps^{\frac15})$ for any {\it fixed} $1\leq i\neq j\leq m$.
To do so, by the uniformity of the bound \eqref{ineq1}, we can assume without loss of generality that $i=1$ and $j=2$.
Secondly, by noting $a=\min I$ and $b=\max I$ the extremities of $I$, we observe that
\[
\Delta_{\eps,1,2} \subset \bigcup_{k=1}^{\lfloor 1/\eps\rfloor} A_{k,\eps}^2\times
I^{m-2},
\]
where $A_{k,\eps}=[a+(b-a)(k-1)\eps,a+(b-a)(k+1)\eps]\cap I$.
As a result,
$$
\E [\mu_N(\Delta_{\eps,1,2})]\leq\sum_{k=1}^{\lfloor 1/\eps\rfloor} \E [Z_N(A_{k,\eps})^{[2]}Z_N(I)^{[m-2]}].
$$
But, for any $k\in\{1,\ldots,\lfloor 1/\eps\rfloor\}$,
\begin{eqnarray*}
&&\E [Z_N(A_{k,\eps})^{[2]}Z_N(I)^{[m-2]}]= \E [Z_N(A_{k,\eps})^{[2]}Z_N(I)^{[m-2]} 
{\bf 1}_{\{Z_N(A_{k,\eps})\geq 2\}}]\\
&\leq& \E[Z_N(I)^{m}{\bf 1}_{\{Z_N(A_{k,\eps})\geq 2\}}]
\leq \E[Z_N(I)^{5m}]^{\frac15}\,\,\Prob(Z_N(A_{k,\eps})\geq 2)^{\frac45}.
\end{eqnarray*}
Lemma \ref{lm3} thus yields
\begin{eqnarray*}
\sup_N \E[\mu_N(\Delta_{\eps,1,2})]= O( \eps^{\frac15}),
\end{eqnarray*}
which in turn implies (\ref{cequilreste}). The proof of Corollary \ref{cor:Rice-trunc} is complete.\qed

\noindent\rule{6cm}{0.4pt}

\noindent{\bf Jean-Marc Aza\"{i}s} \\
Universit\'{e} de Toulouse. \\
jean-marc.azais@math.univ-toulouse.fr\\ 
{\bf Federico Dalmao}\\
Universidad de la Rep\'ublica and Universit\'e du Luxembourg. \\
fdalmao@unorte.edu.uy\\
{\bf Jos\'e R. Le\'on} \\
Universidad Central de Venezuela and INRIA Grenoble. \\
jose.leon@ciens.ucv.ve \\
{\bf Ivan Nourdin} \\
Universit\'e du Luxembourg. \\
ivan.nourdin@uni.lu\\
{\bf Guillaume Poly} \\
Universit\'e de Rennes 1. \\
guillaume.poly@univ-rennes1.fr

\end{document}